\documentclass[12pt]{amsart}
\usepackage{amssymb}
\usepackage{color}

\newcommand{\calT}{{\mathcal T}}

\newcommand{\calR}{{\mathcal R}}
\newcommand{\calO}{{\mathcal O}}

\newcommand{\calS}{{\mathcal S}}

\newcommand{\Q}{{\bf Q}}
\newcommand{\R}{{\bf R}}
\newcommand{\Z}{{\bf Z}}

\newcommand{\GalQ}{{\rm Gal}_{\Q}}

\newtheorem{theorem}{Theorem}[section]
\newtheorem{lemma}[theorem]{Lemma}
\newtheorem{cor}[theorem]{Corollary}
\newtheorem{prop}[theorem]{Proposition}

\theoremstyle{remark}

\begin{document}

\title{Elliptic curves with  Galois-stable cyclic subgroups of order 4}


\author{Carl Pomerance}
\address{Department of Mathematics and
Computer Science, Santa Clara University, Santa Clara, CA 95053,
USA and Mathematics Department, Dartmouth College, Hanover, NH 03755, USA.}
\email{carl.pomerance@dartmouth.edu}

\author{Edward F. Schaefer}
\address{Department of Mathematics and
Computer Science, Santa Clara University, Santa Clara, CA 95053,
USA.}
\email{eschaefer@scu.edu}


\thanks {The first author is grateful for the hospitality of
Santa Clara University and was funded by their Paul R.\ and Virginia P.\
Halmos Endowed Professorship in Mathematics and Computer Science. The authors
are grateful to Paul Pollack and John Voight for several useful conversations.}

\begin{abstract}
Infinitely many elliptic curves over $\Q$ have a Galois-stable cyclic subgroup of order 4.
Such subgroups come in pairs, which intersect in their subgroups of order 2.
%Over the algebraic closure of $\Q$, an elliptic curve has three pairs of cyclic %subgroups of order 4.
%However, an elliptic curve over $\Q$ can only have zero, one, or two pairs of
%Galois-stable cyclic subgroups of order 4.
Let $N_i(X)$ denote the number of elliptic curves over $\Q$ with at least
$i$ pairs of
Galois-stable cyclic subgroups of order 4, and height at most $X$.
In this article we show that
$N_1(X) = c_{1,1}X^{1/3}+c_{1,2}X^{1/6}+O(X^{0.105})$. We also show, as $X\to \infty$, that $N_2(X)=c_{2,1}X^{1/6}+o(X^{1/12})$, the precise nature of the
error term being related to the prime number theorem and the zeros of the
Riemann zeta-function in the critical strip.
%$|N_2(X) - c_{2,1}X^{1/6}|$ $\leq
%X^{1/12}/\exp((\log X)^{3/5+o(1)}))$.
Here, $c_{1,1}= 0.95740\ldots$, $c_{1,2}=- 0.87125\ldots$, and
$c_{2,1}=
0.035515\ldots$ are calculable constants.
Lastly, we show that $N_i(X)=0$ for $i > 2$ (the result being trivial
for $i>3$ given that an elliptic curve has 6 cyclic subgroups of order 4).
%\textcolor{blue}{This article differs from other recent and similar articles in that we need to count lattice points in a hyperbolic shaped region which
%required developing different techniques for dealing with different tails.}
\end{abstract}

\maketitle

\section{Introduction}

Let $E/\Q$ be an elliptic curve and let $\GalQ$ be the absolute Galois group of $\Q$. We say that a cyclic subgroup of $E$ of order 4 is Galois-stable  if it is stable
under the action of $\GalQ$. Note that such subgroups are the kernels of
$\Q$-rational 4-isogenies.
In
Section~\ref{Characterizing}
 we prove that such subgroups
 come in pairs---the two intersect in the same subgroup of order 2.
To be clear, each subgroup in the pair is, itself, Galois-stable.
We will show that a given $E/\Q$
can  have zero, one, or two pairs of such subgroups.
In Section~\ref{Characterizing} we  provide necessary and sufficient
conditions for  $E/\Q$ to have
at least one pair of such subgroups and
similarly for two pairs.

A given $E/\Q$  has a unique model of the form
$y^2=x^3+Ax+B$ where $A, B\in \Z$ and there is no prime
$\ell$ such that $\ell^4 \mid  A$ and $\ell^6\mid B$.
We define the height of $E/\Q$ to be
max$\{|4A^3|,|27B^2|\}$.

In Section~\ref{Characterizing}, we also work out the rational parametrization of elliptic curves $E/\Q$ with
at least one pair, resp.\ two pairs, of Galois-stable cyclic subgroups of order 4.
We can do a change of coordinates and parametrize such elliptic curves
for the model $y^2=x^3+Ax+B$.
%\tb{In Proposition 2.2, we only give $y^2 = x(x^2+\gamma x + \delta^2)$ as a way
%of parametrizing elliptic curves with ``at least one pair''. That is all we
%need for the ``two pairs'' case.
%The parametrizations in \cite[Prop.\ 4.1]{HS} are for $A$ and $B$ from the model  $y^2=x^3+Ax+B$. We never give those parametrizations for $A$ and $B$ in the case
%of ``at least one pair''.
%To get the ``resulting parametrizations'' we let $\hat{x}=x - \frac{\gamma}{3}$ and get
%$\hat{x}^3 + (-\frac{1}{3}\gamma^2 + \delta^2)\hat{x} + (\frac{2}{27}\gamma^3
%- \frac{1}{3}\gamma\delta^2)$
%$=\hat{x}^3 + \delta^2(-\frac{1}{3}t^2 + 1)\hat{x} + \delta^3 (\frac{2}{27}t^3
%- \frac{1}{3}t)$ where $t=\gamma/\delta$. These are the resulting parametrizations
%for $A$ and $B$ that Harron and Snowden are looking for in  \cite[Prop.\ 4.1]{HS}.
%I thought doing the above computation would be a distraction from what we are
%trying to do with the ``two pairs case'' in Section 2.}
 The resulting parametrizations satisfy the conditions of \cite[Prop.\ 4.1]{HS}, from which it follows that
the number of them with height at most $X$ is
of magnitude $X^{1/3}$ in the case of one pair of groups and of magnitude
$X^{1/6}$ in the case of two pairs.

Let $N_i(X)$ count the number of $E/\Q$ of height at most $X$
with at least
$i$ pairs
of Galois-stable cyclic subgroups of order 4.
In \cite{CKV}, among many other interesting results, the authors
show that $N_1(X)$ $=c_{1,1}X^{1/3}+O(X^{1/6})$, with $c_{1,1} = 0.95740\ldots$ a
calculable constant.
This asymptotic plus error estimate
was worked out using the Principle of Lipschitz
for counting lattice points.

In Section~\ref{CountingOnePair} we use Huxley's
improvement on the Principle of Lipschitz
(see \cite{Hu}) and thus get a better error bound.
As far as we know this is the first time this type of strong Lipschitz Principle
has been used in the arithmetic statistics of elliptic curves.
We are then able prove our main theorem for $E/\Q$ with at least one pair
of Galois-stable cyclic subgroups of order 4, namely:
$N_1(X)$ $=
c_{1,1}X^{1/3}$ $+c_{1,2}X^{1/6}$ $+O(X^{0.105})$, with $c_{1,2}=-0.87125\ldots $
a calculable constant.
The $X^{1/6}$ term takes into account lattice points giving
singular curves  $y^2=x^3+Ax+B$ and $N_2(X)$ as well.
We report on a computer experiment
in Section~\ref{Numerical evidence - one pair}  that illustrates the above asymptotic.

In Section~\ref{The parametrization}, we find a 3-variable integer parametrization
for  $E/\Q$ of height at most $X$
with two pairs
of Galois-stable cyclic subgroups of order 4.
%elliptic curves over $\Q$ with two pairs of such subgroups.
We use that parametrization in Section~\ref{Bijections} to find a bijection between the set of elliptic curves
counted by $N_2(X)$ and certain lattice points
in a 3-dimensional region with tails.
In Section~\ref{Useful constants} we compute three constants, which turn out
to be related, that will help us solve our counting problem in each tail.
We present some useful results from analytic number theory in Section~\ref{analytic nt}  and adapt them to the local
restrictions imposed on our counting arguments.

The local restrictions require us to consider subsets of the lattice points in
similar 3-dimensional hyperbolic regions of different sizes. We do this in Section~\ref{The six cases}. In that section we also
cover each region by two sets, each set encompassing a tail, and also
consider the intersection of those two sets.
We then count the appropriate lattice points in each of the three subsets in the covering. In only one of the two tails can we use the
Principle of Lipschitz. We develop new techniques to count the appropriate lattice
points in the other tail.

We assemble those results in Section~\ref{The main theorem} and prove our main
theorem for $E/\Q$ with two pairs of Galois-stable cyclic subgroups of order 4, namely : $|N_2(X) - c_{2,1}X^{1/6}|$ $\leq
X^{1/12}/\exp((\log X)^{3/5+o(1)}))$ as $X\to \infty$,
with $c_{2,1}=
0.035515\ldots$ a calculable constant.  It is to be remarked that the error estimate
here relies on the best known zero-free region of the Riemann zeta function in
the critical strip.  If the Riemann Hypothesis could be assumed, there would
be a corresponding power-saving reduction in the error estimate.
We report on a computer experiment
in Section~\ref{Numerical evidence}  that illustrates the above asymptotic.

Our work is similar to \cite{PPV}. In that article  the authors found that the number of $E/\Q$ of height at most $X$ and with a Galois-stable subgroup of order 3
  is $\frac{2}{3\sqrt{3}\zeta(6)}
X^{1/2} + \eta_1 X^{1/3}\log X + \eta_2 X^{1/3} + O(X^{7/24})$ for
calculable constants $\eta_1$ and $\eta_2$.
They too were unable to use
the Principle of Lipschitz.
This work built on \cite{HS} which studies the number of elliptic
curves over $\Q$, up to a given height bound,  with each possible subgroup of $\Q$-rational
torsion points.

\section{Notation}
\label{Constants}

Let us set down some notation used throughout the article.
For a given $E/\Q$, let $E(\Q)$
denote the Mordell-Weil group, i.e.\ the set of points of $E$ fixed by $\GalQ$.
If $G$ is a group and $g\in G$ we use $\langle g\rangle$ to denote
the cyclic subgroup of $G$ generated by $g$.

There will be several constants defined in this article. In order
that they be easy to find, we define them all here.

Let $\mu(d)$ denote the M\"{o}bius function.
For $n> 1$ we have $\sum_{d=1}^\infty \mu(d)/d^n$ $=\zeta (n)^{-1}$. We will need $\zeta(2)= 1.6449\ldots$
and $\zeta(4)= 1.0823\ldots\, $.

Let
\begin{align*}
\alpha_1:&= ((\sqrt{6}+\sqrt{3})/18)^{1/3} - ((\sqrt{6}-\sqrt{3}/18)^{1/3}
= 0.27314\ldots,\\
\alpha_2 :&=
\frac{1}{2^{1/3}3^{1/2}}= .45824\ldots,\\
i_1 :&= 2\int_{0}^{\alpha_1}
(3u^2+\frac{1}{4^{1/3}})^{1/2}\, du= 0.45804\ldots,\\
i_2 :&= 2\int_{\alpha_1}^{2\alpha_2}
(2u^2+\frac{1}{27^{1/2}u})^{1/2}\, du =1.3591\ldots,\\
i_3 :&= 2\int_{\alpha_2}^{2\alpha_2}(3u^2-\frac{1}{4^{1/3}})^{1/2}\, du
= 0.78093\ldots,\\
i_4 :&= i_1 + i_2 - i_3= 1.03621\ldots,\hbox{  and }\\
c_{1,1}:&=\frac{i_4}{\zeta(4)}
=0.95740\ldots\, .
\end{align*}

For $p_8(v,w)=v^8+14v^4w^4+w^8$, define
\[
s_{0}':=\sum_{\substack{1\leq v<w\\v\,\not\equiv\,w\kern-5pt\pmod{2}}}\frac{1}{\sqrt{p_8(v,w)}}
=.064679\dots\, {\rm and}\,
s_{1}':=\sum_{\substack{1\leq v<w\\2\,\nmid \,vw}}\frac1{\sqrt{p_8(v,w)}}=0.016169\dots,
\] and let
$c_{2,1}:=\frac{16}{2^{1/3}27^{1/2}5\zeta (2)\zeta (4)}
( s_{0}'+ 4s_{1}')= 0.035515\ldots\, .$

Finally, let $c_{1,2}:= -\frac{3\alpha_2}{\zeta(2)} - c_{2,1}= -0.87125\ldots\, .$

% $i_1 = 0.4580481074957766306958424924$
% $i_2 = 1.359101651470524418765380280$
% $i_3 = 0.7809330869234473298162103565$
% $i_4 = 1.036216672042853719645012416$

\section{Characterizing elliptic curves with  Galois-stable cyclic subgroups of order 4}
\label{Characterizing}

%\textcolor{blue}{I changed nothing in Sections 2 - 4.}

Just for the following lemma, we remove our restriction that
our elliptic curve be defined over $\Q$.

\begin{lemma}
\label{4torsioncoordinates}
Let $E$ be an elliptic curve defined over a field of
characteristic other than $2$. Let $E$ be given by $y^2=(x-\rho_1)(x-\rho_2)(x-\rho_3)$.
The four $4$-torsion points doubling to the $2$-torsion point $(\rho_1,0)$
have coordinates
\[
(\rho_1 \pm \sqrt{\rho_1-\rho_2}\sqrt{\rho_1-\rho_3},
\pm \sqrt{\rho_1-\rho_2}\sqrt{\rho_1-\rho_3}(
\sqrt{\rho_1-\rho_2} \pm \sqrt{\rho_1-\rho_3}))
\]
where the first and third $\pm$ must agree.
\end{lemma}

\begin{proof}
The proof is a straightforward computation. This result
appeared in \cite[p.\ 112]{Sc}.
\end{proof}

From now on, our elliptic curves will be defined over $\Q$.
Proposition~\ref{has4isogeny} and Corollary~\ref{isogenypairs}
were independently  proven in Lemma 4.1.3 and the proof
of Proposition 4.1.4 in \cite{CKV}.

\begin{prop}
\label{has4isogeny}
Let $R$ be a point of order $4$ on the elliptic curve $E/\Q$. The following are equivalent.

\begin{itemize}
\item[i)]
The group $\langle R\rangle$ is Galois-stable.

\item[ii)]
For all $\sigma\in \GalQ$, we have $\sigma R = \pm R$.

\item[iii)]
We have $x(R)\in \Q$ (where $x(R)$ denotes the $x$-coordinate of $R$).

\item[iv)]
 $E/\Q$ has a model $y^2=x(x^2+\gamma x + \delta^2)$ with $\gamma\in \Q$,
$\delta\in \Q^\times$, and $\gamma^2-4\delta^2\neq 0$ where $x(R)\in \{ \pm \delta\}$
and $2R=(0,0)$.
\end{itemize}
\end{prop}

\begin{proof}
It is clear that
i) and iii) are each equivalent to ii). Let us prove
ii) implies iv). Assume for all $\sigma\in \GalQ$, that $\sigma R = \pm R$.
Then
$\sigma (2R)=2R$ so $2R\in E(\Q)[2]$. So $E$ has a model $y^2=x(x^2+\gamma x + \epsilon)$
with $\gamma, \epsilon\in \Q$ and $2R=(0,0)$. Since the cubic in $x$ can not have
repeated roots we have $\epsilon\in \Q^\times$ and $\gamma^2-4\epsilon\neq 0$. From
Lemma~\ref{4torsioncoordinates}, $x(R) \in \{ \pm \sqrt{\epsilon}\}$. Since ii) implies
iii), we have $\epsilon = \delta^2$ for some $\delta\in \Q^\times$.

Now we prove iv) implies i). It is a straightforward calculation that  the
two points with $x$-coordinate $\delta$, the point
$(0,0)$, and the $\calO$-point of the elliptic curve form a Galois-stable
cyclic group of order 4. The same is true if we replace $\delta$ by $-\delta$ in the previous
sentence.\end{proof}

%\begin{cor}
%\label{OnePairParametrization}
%The elliptic curve $E/\Q$ has a $\Q$-rational 4-isogeny if and only if
%$E$ has a model
%$y^2=x(x^2+\gamma x + \delta^2)$ with $\gamma\in \Q$,
%$\delta\in \Q^\times$.
%\end{cor}

%\begin{proof}
%This  follows from the proof of Proposition~\ref{has4isogeny}.
%\end{proof}

\begin{cor}
\label{isogenypairs}
The Galois-stable cyclic subgroups of order $4$ of $E/\Q$, with model
$y^2 = x(x^2+\gamma x + \delta^2)$, come in pairs
where the $x$-coordinate of the generators of one subgroup is the negative
of the $x$-coordinate of the generators of the other.
The intersection of the two groups of order $4$ is the subgroup of each of order $2$.
The point generating this subgroup of order $2$  is $\Q$-rational.
\end{cor}

\begin{proof}
This  follows from the proof of Proposition~\ref{has4isogeny}.
\end{proof}

%Let us look at Corollary~\ref{isogenypairs} from another point of view.
% Assume $E/\Q$ has such a subgroup, which is generated by the point $P$. Assume %$\langle P,Q\rangle = E[4]$. We can represent
% the action of $\GalQ$ on $E[4]$ as a subgroup of the group of matrices (on the % %ordered pair of generators $(P,Q)$) of the form
 %\[
 %\left[
% \begin{array}{cc}
% a & b \\ 0 & c \end{array}\right] \]
% where $a, c\in \{ 1, 3 \}$ and $b\in \Z_4$.
% Note that each such matrix sends (the point of order 4) $P+2Q$ to $\pm (P+2Q)$. So % % from Proposition~\ref{has4isogeny}, $E$ has a Galois-stable subgroup generated
%  by $P+2Q$ as well.

 Note the number of points of order 4 on $E$ over the algebraic closure of $\Q$ is 12. So $E$ has six cyclic subgroups of order 4.
 Thus, there are
 at most three pairs of Galois-stable subgroups. In
Corollary~\ref{hasthreepairsofisogenies}, we will show
it is impossible to have three pairs of Galois-stable cyclic subgroups of order 4.
Note, when we refer to a pair of Galois-stable subgroups, we mean that each subgroup within the pair is itself Galois-stable.

\begin{prop}
\label{hastwopairsofisogenies}
The elliptic curve $E/\Q$ has two pairs of  Galois-stable cyclic subgroups of order $4$
if and only if $E/\Q$ has exactly one model of the form
\begin{equation}
\label{ModelTwoPairs}
y^2 = x(x-r)\Big(x-r \Big( \frac{1-\tau^2}{1+\tau^2}\Big)^2\Big)\end{equation}
with $r$ a squarefree positive integer and  $\tau\in \Q$, with $0< \tau < 1$.
\end{prop}
\begin{proof}
Assume $E$ has two pairs of such subgroups, namely
$\{ G_1, G_2\}$ and $\{ G_3, G_4\}$. Assume $G_1\cap G_2 = \langle
T_1\rangle$ and $G_3\cap G_4 = \langle T_2\rangle$ where $T_1$, $T_2$ are points
of order 2. There are exactly four points of order 4 which double to $T_1$ and those
are the generators of $G_1$ and $G_2$. So $T_2\neq T_1$.
We see $T_1, T_2$ generate $E[2]$ (the 2-torsion subgroup of $E$), which from Corollary~\ref{isogenypairs} is contained
in  $E(\Q)$.
Combining this with the fact that $E$ has at least one pair of such subgroups,
we get, from Proposition~\ref{has4isogeny}, that $E$ has a model $y_1^2 =
x_1(x_1-r_1)(x_1-r_1\beta_1^2)$ with
$r_1$ a nonzero integer
 and $\beta_1 \in \Q\setminus
\{ -1, 0, 1\}$ (so that the cubic does not have a double root).

Without loss of generality, assume $T_1 = (0,0)$ and $T_2 = (r_1,0)$. We make the change of
variables $x_2: = x_1 - r_1$ and get the model for $E$ given by
$y_1^2 = x_2(x_2+r_1)(x_2+r_1 ( 1-\beta_1^2))$ or
$y_1^2 = x_2(x_2^2 + r_1 (2-\beta_1^2)x_2+r_1^{2} (1-\beta_1^2))$.
Note $T_2$ has coordinates $(x_2,y_1) = (0,0)$. Given that $E$ has a pair of
such subgroups, each containing $T_2$, we see from
 Proposition~\ref{has4isogeny} that $1-\beta_1^2 = \eta_1^2$ for some
 $\eta_1\in \Q^\times$. We know from a famous parametrization of the unit circle
 that $\eta_1 = (1-\tau_1^2)/(1+\tau_1^2)$ for some $\tau_1\in \Q$,
 with $0 < \tau_1 < 1$ (we remove $\tau_1=0, 1$ from consideration
 as the cubic has a double root in those cases).
  There is a unique choice of $\tau_1$ since $(1-\tau_1^2)/(1+\tau_1^2)$
  is monotonic on $0 < \tau_1 < 1$.

 There is a unique change in variables, by scaling $x_2$ and $y_1$
 to $x_3$ and $y_2$, respectively, so that $E/\Q$ has form
 $y_2^2 = x_3(x_3 - r_2)(x_3-r_2 ((1-\tau_1^2)/(1+\tau_1^2))^2)$
with $r_2$ a non-zero squarefree integer.
 Namely, let $r_2$ be the unique nonzero squarefree integer
 such that there exists $\gamma_1\in \Q^\ast$ such that
 $r_1 =  \gamma_1^2 r_2$. Then we let $x_3 := x_2/\gamma_1^2$ and
 $y_2 := y_1/\gamma_1^3$.

 To look for other possible models of $E/\Q$, as specified in the statement of this proposition,
 we translate each of the two non-zero roots of $x_3(x_3 - r_2)(x_3-r_2 ((1-\tau_1^2)/(1+\tau_1^2))^2)$ to zero.
 First, let $x_4 := x_3 - r_2$.
 We get
 \begin{equation}
 \label{2ndModel}
   y_2^2 = x_4 (x_4+r_2)\Big(x_4 + r_2 \Big( \frac{2\tau_1}{1+\tau_1^2}\Big)^2 \Big)\; {\rm or}\;
   y_2^2 = x_4 (x_4 - r_3)\Big(x_4 - r_3 \Big(
   \frac{1-\tau_2^2}{1+\tau_2^2}
      \Big)^2 \Big)
\end{equation}
 where $r_3:=-r_2$ and $\tau_2 := \frac{1-\tau_1}{1+\tau_1}$.
 This is of the form of \eqref{ModelTwoPairs}.

 Second, let $x_5:= x_3 - r_2 ((1-\tau_1^2)/(1+\tau_1^2))^2$.
 When we make the substitution, the coefficient of $x_5$ in the model
 for the elliptic curve is
 $-(2r_2(\tau_1^3-\tau_1)/(\tau_1^2+1)^2)^2$ and hence
 can not be a square.  So from Proposition~\ref{has4isogeny}, we do not get
 a third model of the form specified in \eqref{ModelTwoPairs}.

 Given that $r_2$ is non-zero, exactly one of $r_2$ and $r_3$
 is positive --- that is the one we choose for our model.

 The proof of the reverse implication is a straightforward computation using
 Proposition~\ref{has4isogeny}, its proof, and Corollary~\ref{isogenypairs}. \end{proof}

\begin{cor}
\label{hasthreepairsofisogenies}
If $E/\Q$ has two  pairs of Galois-stable cyclic subgroups of order $4$, then each cyclic subgroup of order $4$ in the remaining
pair is defined over $\Q(i)$, but not $\Q$.
\end{cor}

\begin{proof}
We  choose a model for $E/\Q$ of the
form given in \eqref{ModelTwoPairs}.
We then translate the 2-torsion point
$(r((1-\tau^2)/(1+\tau^2))^2,0)$ to $(0,0)$.
From the Proof of Proposition~\ref{hastwopairsofisogenies},
the coefficient of $x$ in this latter model is then $-(2r(\tau^3-\tau)/(\tau^2+1)^2)^2$, which is $-1$ times a square. We see, from a generalization of
Proposition~\ref{has4isogeny} replacing $\Q$ by $\Q(i)$, that the remaining two such subgroups  are
defined over $\Q(i)$, but not $\Q$.
\end{proof}

\section{Counting elliptic curves with at least one pair of Galois-stable cyclic subgroups of order 4}
\label{CountingOnePair}

There is a bijection between pairs $(A,B)$ with $A,B\in \Z$,
$4A^3+27B^2\neq 0$, for which there is
no prime $\ell$
such that $\ell^4 \mid A$ and $\ell^6 \mid B$, and elliptic curves over $\Q$
(up to isomorphism) sending the pair $(A,B)$ to the elliptic curve with
model $y^2=x^3+Ax+B$. So we will count such pairs $(A,B)$
for which $y^2=x^3+Ax+B$ has
height at most $X \geq 1$ and
at least one pair of Galois-stable cyclic subgroups of order 4.

The results in the following paragraphs up to Proposition~\ref{HowToCountOnePair} can be found in
\cite[pp.\ 18-20]{CKV}.
From Corollary~\ref{isogenypairs}, $E/\Q$ has a pair of Galois-stable cyclic
subgroups of order 4 if and only if there exists a rational number
$b$, a root of $x^3+Ax+B$, and when we replace $x$ by $x+b$
the coefficient of $x$ is a square.
Note that since $x^3+Ax+B$ is monic and $A,B\in \Z$, we can replace the word rational
in the previous sentence by the word integer. If we replace $x$ by $x+b$
in $x^3+Ax+B$ we get $x^3 + 3bx^2 + (3b^2+A)x + (b^3+Ab+B)$.
So we are looking for integers $b$ for which there is an integer $a$ such
that $3b^2+A=a^2$. We now have a map sending the integer pair $(a,b)$ to the integer
pair
$(A,B)=(a^2-3b^2,a^2b-b^3)$.

Recall the height of this elliptic curve  is
max$\{ |4A^3|,|27B^2|\}$.
Given a height bound $X\geq 1$, we define a region,
which we denote
\[
\calR'_1(X):= \{ (a,b)\in \R\times \R \; : \; 4|a^2-3b^2|^3\leq X\;
{\rm and}\; 27|a^2b-2b^3|^2\leq X\}.
\]
Most lattice points
$(a,b)\in \calR'_1(X)$
give rise to a pair $(E/\Q, \calT)$ where $\calT$ is a pair of
Galois-stable cyclic subgroups of $E$ of order 4.
The exceptions are those lattice points giving singular curves. That occurs when
$0$ $=4A^3+27B^2$ $=4(a^2-3b^2)^3 + 27(a^2b-2b^3)$ $=4a^6 - 9b^2a^4$
$=a^4(4a^2-9b^2)$. So we must remove from consideration the points in
$\calR'_1(X)$ on the singular locus: $a=0$ and $a=\pm 3b/2$.

An easy exercise shows that there is a prime $\ell$ such that $\ell^4|A$ and
$\ell^6|B$ if and only if $\ell^2|a$ and $\ell^2|b$.
Let $\calR_1(X)$ be the subset of $\calR'_1(X)$ for which $a\geq 0$.
If we remove from $\calR_1(X)$ the lattice points $(a,b)$ on the singular locus, and those for which
there is a prime $\ell$ such that $\ell^2|a$ and $\ell^2|b$,
then we get an
injection from the remaining lattice points  to pairs
$(E/\Q,\calT)$, where $\calT$ is a pair of
Galois-stable cyclic subgroups of order 4.

\begin{prop}
\label{HowToCountOnePair}
The function $N_1(X)$ can be computed by subtracting the number of $E/\Q$ of height at most
$X$ with two pairs of Galois-stable cyclic subgroups of
order $4$  from the number of lattice points in $\calR_1(X)$, not on $a=0$ or $a=3|b|/2$, for which there is no prime $\ell$ such that $\ell^2\mid a$ and
$\ell^2\mid b$.
\end{prop}

\begin{proof}
This follows from
\cite[pp.\ 18-20]{CKV} and
Corollary~\ref{hasthreepairsofisogenies}.
\end{proof}

With the above proposition and an important result of Huxley \cite{Hu}, we
can estimate the count.  The Huxley result asserts that if $\calR$ is a compact,
convex region in the plane of area $A$ and with the boundary being piecewise
smooth and with the curvature on each piece non-zero and 3 times continuously
differentiable, then the number of lattice points (i.e., integer points) contained in
$\calR$ scaled by a large factor $r$ is
$Ar^2+O(r^{0.63})$.  (The actual error bound is $r^{131/208}(\log r)^{O(1)}$,
see \cite[p.~592]{Hu}.)
We would like to apply this theorem to $\calR'_1(X)$, which is a compact
region with piecewise smooth boundary and non-zero curvature scaled by
a factor $X^{1/6}$, however, the boundary segments are concave with
respect to the interior, not convex.  We can nevertheless apply the Huxley
theorem by recognizing the region as the union of differences of convex
sets with boundaries being piecewise smooth and with nonzero curvature.
(A simple example is a lune, which is the difference of two convex sets.)
%All of the 8 boundary segments of $\calR'_1(X)$ have non-zero curvature and are concave with
%respect to the interior.
% We can cover the
%complement of  $\calR'_1(X)$ in a circle enclosing $\calR'_1(X)$ by a finite
%number of convex regions,
%all of whose boundary segments are curves with non-zero curvature and
%are multiply differentiable.
%Consider $\calR'_1(X)$ as a subset of a circular disc, where we can view the
%complement dissected into a finite number of regions that are either convex
%or the difference of convex regions, all boundary segments being smooth
%and with non-zero curvature.
Since the area of $\calR'_1(X)$ is  $2i_4X^{1/3}$ (see Section~\ref{Constants}), it follows
%from \cite[p.\ 592]{Hu}
that the number of lattice points in $\calR'_1(X)$
is equal to $2i_4X^{1/3}+O(X^{0.105})$.
%(The actual exponent is $131/1248+\epsilon$ for any fixed $\epsilon>0$.)
%The area of $\calR'_1(X)$ is $2i_4X^{1/3}$
%as $i_4X^{1/3}$ gives the area of $\calR_1(X)$ (see Section~\ref{Constants}).

We next want to count the number of lattice points in
$\calR'_1(X)$ for which there is no prime $\ell$
such that $\ell^2 \mid a$ and $\ell^2\mid b$. For a given positive integer
$d$, the number of lattice points $(a,b)\in \calR'_1(X)$ for which
$d^2\mid a$ and $d^2\mid b$ is equal to the number of lattice points in $\calR'_1(X)$
when it has been scaled down by $d^2$ in both dimensions.  Since the number
of lattice points in this scaled down region is $2i_4X^{1/3}/d^4+O((X^{1/6}/d^2)^{0.63})$,
the number of lattice points in $\calR'_1(X)$ for which there is no prime $\ell$
such that $\ell^2 \mid a$ and $\ell^2\mid b$ is
\[
\sum_{d\le \sqrt{\alpha_2}X^{1/12}} \Bigg(\frac{2i_4\mu(d)}{d^4}X^{1/3} + O\Big( \frac{X^{.105}}{d^{1.26}}\Big)\Bigg)
=\frac{2i_4}{\zeta(4)}X^{1/3}+O(X^{.105}).
\]
Here, we used that $\sum_d\mu(d)/d^4=1/\zeta(4)$ and that the error in truncating
this sum at $X^{1/12}$ is, by an elementary argument, $O(1/X^{1/4})$.

As lattice points on $a=0$ give singular curves $y^2=x^3+Ax+B$, we are ultimately
interested in lattice points with  $a > 0$.
The lattice points on $a=0$ in $\calR'_1(X)$ are the points $(0,b)$
with $|b|\leq \alpha_2 X^{1/6}$.
For such points, there is a prime $\ell$ such that $\ell^2\mid a$ and $\ell^2\mid b$
if and only if $\ell^2\mid b$. So the number of lattice points
in $\calR'_1(X)$ on $a=0$ for which there is no
prime $\ell$ such that $\ell^2\mid a$ and $\ell^2\mid b$ is
$\frac{2\alpha_2}{\zeta(2)}X^{1/6}+O(X^{1/12})$,
arguing as above.
%(see Proposition~\ref{lem:useful}, part (2)).
%\tb{This brings up the question of where the section ``Facts from analytic number
%theory'' should go. Despite pre-referencing it here, I think it should stay where
%it is since much of it is adapted for the $3\mid r$, $3\nmid r$ cases
%which are unique to counting elliptic curves with two pairs of \ldots }
Thus, the number of lattice points in $\calR_1(X)$ with $a>0$ and for which
there is no
prime $\ell$ such that $\ell^2\mid a$ and $\ell^2\mid b$ is $\frac{i_4}{\zeta(4)}X^{1/3}
- \frac{\alpha_2}{\zeta(2)}X^{1/6}+O(X^{0.105})$.

The remaining pairs $(a,b)$ giving singular curves $y^2=x^3+Ax+B$ are those
on $a=3|b|/2$. We have already removed those for which $\ell^2\mid a$ and $\ell^2\mid b$.
The part of $a=3|b|/2$ in $\calR_1(X)$ is that for which $|b|\leq 2\alpha_2 X^{1/6}$.
The lattice points on $a=3|b|/2$
are of the form $(3|k|,2k)$ for $k\in \Z$ with $|k|\leq
\alpha_2 X^{1/6}$. So there are $2\alpha_2 X^{1/6}+O(1)$ of them.
If $\ell$ is prime then $\ell^2\mid 2k$ and $\ell^2\mid 3|k|$ if and only if $\ell^2 | k$.
So the number of lattice points on $a=3|b|/2$ in $\calR_1(X)$ for which there is no prime $\ell$ such that $\ell^2|a$ and $\ell^2|b$ is
$\frac{2\alpha_2}{\zeta(2)}X^{1/6}+O(X^{1/12})$.

We show in Theorem~\ref{Count2pairs} that the number of $E/\Q$ of height at most $X$
with two pairs of Galois-stable cyclic subgroups of order 4 is $c_{2,1}X^{1/6}+O(X^{1/12})$.
Recall $c_{1,2}=-\frac{3\alpha_2}{\zeta(2)} - c_{2,1}= -0.87125\ldots $
 (see Section~\ref{Constants}).

\begin{theorem}
\label{thm:MainOnePair}
The number of $E/\Q$ of height at most $X$ with at least one pair of Galois-stable
cyclic subgroups is $N_1(X) = c_{1,1}X^{1/3}+c_{1,2}X^{1/6}+O(X^{0.105})$.
\end{theorem}
\begin{proof}
This follows from Proposition~\ref{HowToCountOnePair} and the computations above.
\end{proof}

\subsection{Numerical evidence - one pair}
\label{Numerical evidence - one pair}

In this section we give numerical evidence for Theorem~\ref{thm:MainOnePair}.
Here we use $N_1(X)$ to denote the number of $E/\Q$ of height
at most $X$ with two pairs of Galois-stable cyclic subgroups of order 4 (as counted in Section~\ref{Numerical evidence}) subtracted from the number
of pairs $(a,b)\in \calR_1(X)$, for which the corresponding $A, B$ satisfy
$4A^3+27B^2\neq 0$, and for which there
is no prime $\ell$ such that $\ell^2|a$ and $\ell^2|b$.
In the table below we present $N_1(X)$ and
$N_1(X) - c_{1,1}X^{1/3}-c_{1,2}X^{1/6}$
(which we round to one digit past the decimal point) for various values of $X$.
We have
$c_{1,1}=0.957400377047\ldots$
and
$c_{1,2}=-0.871250852030\ldots\, $.
\medskip

\centerline{
\renewcommand{\arraystretch}{1.5}
\begin{tabular}{|l|r|r|}\hline
$X$ & $N_1(X)\hskip.4cm$ & $N_1(X) - c_{1,1}X^{1/3}-c_{1,2}X^{1/6}$ \\
\hline
$10^{18}$ & 956574 & $44.9\hskip.8in$ \\
$10^{21}$ & 9571217 & $-31.6\hskip.8in$  \\
$10^{24}$ & 95731445 & $119.8\hskip.8in$ \\
$10^{27}$ & 957372610 & $-215.7\hskip.8in$ \\
$10^{30}$ & 9573916722 & $76.6\hskip.8in$  \\
\hline
\end{tabular}
}

%\centerline{
%\renewcommand{\arraystretch}{1.5}
%\begin{tabular}{|l|r|r|r|}\hline
%$X$ & $N_1(X)\hskip.4cm$ & $N_1(X) - c_{1,1}X^{1/3}$ & $(N_1(X) - c_{1,1}X^{1/3})/X^{1/6}$
%\\
%\hline
%$10^{18}$ & 956574 & $-826.4$ & $-0.82638$ \\
%$10^{21}$ & 9571217 & $-2786.8$ & $-0.88126$ \\
%$10^{24}$ & 95731445 & $-8592.7$ & $-0.85927$ \\
%$10^{27}$ & 957372610 & $-27767.0$ & $-0.87807$ \\
%$10^{30}$ & 9573916722 & $-87048.5$ & $-0.87049$ \\
%\hline
%\end{tabular}
%}

\section{Counting elliptic curves with two pairs of Galois-stable cyclic subgroups of order 4}

 \subsection{The parametrization}
 \label{The parametrization}

In this section, we want to find an integer parametrization of elliptic curves
with two pairs of Galois-stable cyclic subgroups of order 4.
 From Proposition~\ref{hastwopairsofisogenies}
such an elliptic curve  has a model of the form in
\eqref{ModelTwoPairs}
with $r$ a squarefree positive integer, $\tau\in \Q$ and $0< \tau < 1$.

We want to find a model for such an elliptic curve of the form $y^2=x^3+Ax+B$
with $A,B\in \Z$, and for which there is no prime $\ell$ such that
$\ell^4 \mid A$ and $\ell^6 \mid  B$.
In the model given by \eqref{ModelTwoPairs}, we replace $\tau$ by $v/w$ where $v, w$
are variables representing relatively prime integers
with $1 \leq v < w$.
Then we replace $x$ by
$(x+6r(v^4+w^4))/(9(v^2+w^2)^2)$
and $y$ by
$y/(27 (v^2+w^2)^3)$.
We  get
$y^2=x^3+Ax+B$ where $
A = - 27r^2(v^8+14v^4w^4+w^8)$ and $B= 54r^3( v^{12}-33v^8 w^4-33v^4 w^8+w^{12})$.
 Set
 \[
 p_8(v,w):=v^8+14v^4w^4+w^8,\quad p_{12}(v,w):=v^{12}-33v^8 w^4-33v^4 w^8+w^{12}.
 \]
We now need to ensure there is no prime $\ell$
such that
$\ell^4 \mid A$ and $\ell^6 \mid  B$.

\begin{lemma}
\label{p8andp12gcd}
Let $1 \leq v < w$ with gcd$(v,w)=1$. If $2\mid v w$, then $\gcd(p_8(v,w),p_{12}(v,w))=1$.
If $2\nmid vw$, then
$16 \mid  p_8(v,w)$, $64\mid p_{12}(v,w)$, and $\gcd(\frac{1}{16}p_8(v,w),
\frac{1}{64}p_{12}(v,w))=1$.
\end{lemma}

\begin{proof}
The resultant of $p_8$ and $p_{12}$ is $2^{40}3^{12}w^{96}$.
Thus, in order for $\ell\mid p_8(v,w)$ and $\ell\mid p_{12}(v,w)$ we need
$\ell = 2$, $\ell = 3$, or $\ell > 3$ with $\ell \mid w$, in which case $\ell\mid v$, which is impossible
since gcd$(v,w)=1$.
We check all nine cases for $v$ and $w$ mod 3 and note that $3\mid p_8(v,w)$ and
$3\mid p_{12}(v,w)$ if and only if $3\mid v$ and $3\mid w$.

We check all four cases for $v$
and $w$ mod 2 and note that $2\mid p_8(v,w)$ and $2\mid p_{12}(v,w)$ if and only
if $v\equiv w~({\rm mod}\, 2)$. Since gcd$(v,w)=1$ it suffices
to consider the case $v\equiv w\equiv 1~({\rm mod}\, 2)$.
A straightforward exercise shows then that
$p_8(v,w)\equiv 16~({\rm mod}\, 64)$ and
$p_{12}(v,w)\equiv -64~ ({\rm mod}\, 256)$.
Therefore
gcd$(\frac{1}{16}p_8(v,w),
\frac{1}{64}p_{12}(v,w))=1$.
\end{proof}

Define the functions $A(r,v,w)$ and $B(r,v,w)$ in the following
way.

\begin{itemize}
\item[(i)] $A(r,v,w):= -27r^2p_8(v,w)$, $B(r,v,w) := 54r^3 p_{12}(v,w)$ when
$3\nmid r$
 and $v\not\equiv w\,({\rm mod}\, 2)$,

\item[(ii)]  $A(r,v,w):= -\frac{1}{3}r^2p_8(v,w)$, $B(r,v,w) := \frac{2}{27}r^3 p_{12}(v,w)$ when
$3 \mid  r$ and $v\not\equiv w\,({\rm mod}\, 2)$,

\item[(iii)]
$A(r,v,w) := -\frac{27}{16}r^2p_8(v,w)$, $B(r,v,w) := \frac{27}{32}r^3 p_{12}(v,w)$ when
$3\nmid r$ and $2\,\nmid vw$, and

\item[(iv)]
$A(r,v,w) := -\frac{1}{48}r^2p_8(v,w)$, $B(r,v,w) := \frac{1}{864}r^3 p_{12}(v,w)$ when
$3 \mid  r$ and $2\nmid vw$.
\end{itemize}

\begin{lemma}
\label{TwoPairsIntegerModel}
Let $E$ be the elliptic curve with two pairs of Galois-stable cyclic subgroups of order $4$ given by the model in \eqref{ModelTwoPairs}.
Let $\tau = v/w$ with $v,w\in \Z$, $1\leq v < w$, and gcd$(v,w)=1$.
The unique model for $E$ of the form $y^2=x^3+Ax+B$
with $A,B\in \Z$ and for which there is no prime
$\ell$ such that $\ell^4 \mid  A$ and $\ell^6\mid B$ is given by
$A = A(r,v,w)$ and $B=B(r,v,w)$.
\end{lemma}

\begin{proof}
At the beginning of this section
%Section~\ref{CountingTwoPairs}
we found that $E$ has an integer model
$y^2=X^3+Ax+B$ with $A = -27r^2p_8(v,w)$ and $B=54r^3p_{12}(v,w)$.

Now we consider the cases where there is a prime $\ell$
with $\ell^4 \mid  A$ and $\ell^6 \mid  B$.
Given Lemma~\ref{p8andp12gcd} and that $r$ is squarefree,
$A = -27r^2p_8(v,w)$, and $B=54r^3p_{12}(v,w)$, the only possible primes $\ell$ are
$\ell = 2, 3$.
We see that $3^4\mid  A$ and $3^6\mid B$ if and only
if $3\mid r$. If $3\mid r$ then replace $A$ by $A/3^4$
and replace $B$ by $B/3^6$.
From Lemma~\ref{p8andp12gcd}, we have that $2^4 \mid  A$ and $2^6 \mid  B$ if and only if
$2\,\nmid vw$, and in this case $A/2^4$ and $B/2^6$ are odd.
 In this case replace
$A$ by $A/2^4$ and $B$ by $B/2^6$.

Now there is no prime $\ell$ such that $\ell^4 \mid  A$
and $\ell^6 \mid  B$.
\end{proof}

Recall that the height of the elliptic curve $y^2=x^3+Ax+B$, where $A, B\in \Z$,
 is at most $X$ if and only if
$|A|\leq A_b$ and $|B|\leq B_b$ where $A_b=\frac{1}{4^{1/3}}X^{1/3}$ and $B_b=\frac{1}{27^{1/2}}X^{1/2}$.

 \begin{lemma}
 \label{ABoundGivesBBound}
 Let $r, w$ be positive integers, $v$ be a non-negative integer,
$v\leq w$, and let $\theta$ be a positive real number.
 If $|-27\theta^4 r^2p_8(v,w)|\leq A_b$
 then $|54\theta^6 r^3p_{12}(v,w)|\leq B_b$.
 \end{lemma}

 \begin{proof}
 Let
 \[
 U=4^{1/3}|27\theta^4r^2p_8(v,w)|,\quad
 V=27^{1/2}|54\theta^6r^3p_{12}(v,w)|.
 \]
 It will suffice to show that $U^{3/2}\ge V$, for if $U\le X^{1/3}$, this will
 imply that $V\le X^{1/2}$.  Since
 \[
 (4^{1/3}27\theta^4r^2)^{3/2}=27^{1/2}54\theta^6r^3,
 \]
 it will thus suffice to show that
$ p_8(v,w)^3\ge|p_{12}(v,w)|^2$.
 However, letting $x=(v/w)^4$, we have
\[
 p_8(v,w)^3/w^{24}=x^6+42x^5+591x^4+2828x^3+591x^2+42x+1,
\]
 while
\[
 |p_{12}(v,w)|^2/w^{24}=
 x^6-66x^5+1023x^4+2180x^3+1023x^2-66x+1.
 \]
We verify that
 $p_8(v,w)^3/w^{24}-p_{12}(v,w)^2/w^{24}=108x(x-1)^4\ge0$,
completing the proof.
     \end{proof}

\subsection{Bijections used in counting}
\label{Bijections}

We gave a homogeneous integer parametrization in Lemma~\ref{TwoPairsIntegerModel} for
$E/\Q$ with two pairs of Galois-stable cyclic subgroups of order 4.
A dehomogenization of this parametrization to a rational parametrization
satisfies the conditions of \cite[Prop.\ 4.1]{HS}. It follows that if $\calS_{\rm pairs}$ is the set of integer pairs
$(A,B)$ such that there is no prime
$\ell$ such that $\ell^4 \mid  A$ and $\ell^6 \mid  B$,
and
$y^2=x^3+Ax+B$ is an elliptic curve of height at
most $X$ with two pairs of Galois-stable cyclic subgroups of order 4, then $\# \calS_{\rm pairs}$ is of order of magnitude $X^{1/6}$.
We will prove an even stronger result.

Let $\calS_{E}$ denote the set of $E/\Q$ (up to isomorphism),
with two pairs of Galois-stable cyclic subgroups of order 4 and height
at most $X$.
 In this section we determine an asymptotic plus error estimate for the size of $\calS_E$. There is a bijection from $\calS_{\rm pairs}$ to $\calS_{E}$
sending $(A,B)$ to the elliptic curve $y^2=x^3+Ax+B$.

Let $\calS_{r,v,w}$ denote the set of triples
$(r,v,w)$ with $r$ a squarefree positive
integer, $v, w\in \Z$ with $1\leq v < w$, gcd$(v,w)=1$,
and $|A(r,v,w)|\leq A_b$.

\begin{prop}
\label{BijectionAVWandTwoPairs}
There is a bijection between $S_{r,v,w}$ and $\calS_{\rm pairs}$
 sending
$(r,v,w)$ to $(A(r,v,w),$ $B(r,v,w))$.
\end{prop}

\begin{proof}
This follows from Lemmas~\ref{TwoPairsIntegerModel} and \ref{ABoundGivesBBound}, where we take $\theta = 1, \frac{1}{3}, \frac{1}{2}, \frac{1}{6}$ in Cases (i) - (iv)
(listed before Lemma~\ref{TwoPairsIntegerModel}) respectively.
\end{proof}

To determine the size of the set $\calS_{E}$ we will instead determine the
identical size of the set $\calS_{r,v,w}$. We see, from the four cases
presented immediately before Lemma~\ref{TwoPairsIntegerModel}, that we want
to count the number of triples $(r,v,w)$ of positive integers with $r^2 p_8(v,w)\leq \eta X^{1/3}$ for $\eta$ a constant, $r$ squarefree, and $ v < w$.

\subsection{Useful constants}
\label{Useful constants}

%\textcolor{blue}{This is a new paragraph.}
In $(r,v,w)$-space, the region with all variables positive and $r^2 p_8(v,w)\leq \eta X^{1/3}$ (for $\eta <1$ a constant) can be considered to have two tails --- one for $r$ sufficiently large and
one for $p_8$ sufficiently large. Our counting argument of valid triples $(r,v,w)$ requires us to consider those tails differently. In particular, the Principle of
Lipschitz,
commonly used in similar counting arguments, is useful in only one of the tails.  Certain constants arise within the counting argument for each tail and we compute them here. As the function $A(r,v,w)$ depends on the residue classes of $v$ and
$w$ modulo 2, we will need to take that into consideration when dealing with these
constants.
For $i,j\in \{ 0, 1\}$, with $(i,j)\neq (0,0)$, let
\begin{equation}
\label{TijDefinition}
T_{ij} = \{ (v,w)\in \Z^2\; | \; v\equiv i~({\rm mod}\, 2),\, w\equiv j~({\rm mod}\, 2)\}.\end{equation}
%\textcolor{blue}{Carl, Lemma~\ref{lem:tail}, Corollary~\ref{cor:coprime}, and
% Corollary~\ref{cor:PL} have all been adapted to the residue classes.}

 \begin{prop}
 \label{lem:tail}
 Let
 \[
 \alpha_3=\int_0^1\frac1{\sqrt{p_8(u,1)}}\,du=0.691002044642207\dots
 \]
 and let
 \[
 \alpha_4=\int_1^2\int_{g(t)}^1\frac1{\sqrt{p_8(u,1)}}\,du\,dt
% =0.0890402893565981\dots,
%=0.3802755797109349\dots,
=0.1223644572102272\dots,
 \]
 where $g(t)=((t^4+48)^{1/2}-7)^{1/4}$.
 Then for $y\ge1$,
 \[
 \sum_{\substack{0<v<w\\p_8(v,w)>y\\(v,w)\in T_{i,j}}}\frac1{\sqrt{p_8(v,w)}}=\frac{\alpha_3+\alpha_4}{8y^{1/4}}+O\Big(\frac1{y^{3/8}}\Big).
 \]
 \end{prop}
 \begin{proof}
 For each fixed value of $w>y^{1/8}$ and all $v$ with $1\le v<w$, we have $p_8(v,w)>y$.  Further
 \[
 \sum_{\substack{1\leq v<w\\v\,\equiv\,i\kern-5pt\pmod2}}\frac1{\sqrt{p_8(v,w)}}=
\frac1{w^4}\sum_{\substack{1\le v<w\\v\,\equiv\,i\kern-5pt\pmod2}}
\frac1{\sqrt{p(v/w,1)}}=\frac1{2w^3}
\sum_{\substack{1\le v<w\\v\equiv\,i\kern-5pt\pmod2}}\frac1{\sqrt{p(v/w,1)}}\cdot\frac2w.
% \sum_{\ell = 0}^{\frac{w-2}{2}}\frac{1}{\sqrt{(2\ell+i)^8+14(2\ell+i)^4w^4+w^8}}
\]
 %\[
% = \sum_{\ell = 0}^{\frac{w-2}{2}}\frac{1}{w^4\sqrt{(\frac{2\ell+i}{w})^8 +
%14(\frac{2\ell+i}{w})^4+1}}
% =\frac{1}{2w^3}\sum_{\ell=0}^{\frac{w-2}{2}}\frac{1}{\sqrt{(\frac{2\ell +i}{w})^8 +
 %14(\frac{2\ell+i}{w})^4+1}   }\left(\frac{2}{w}\right).
% \]
   This last sum is a Riemann sum for the integral $\alpha_3$, and by
 monotonicity is equal to $\alpha_3+O(1/w)$.  Thus,
 \[
  \sum_{\substack{1\leq v<w\\v\,\equiv\, i~({\rm mod}\, 2)}}\frac1{\sqrt{p_8(v,w)}}=\frac{\alpha_3}{2w^3}+O\left(\frac1{w^4}\right).
  \]
  Summing this expression for $w>y^{1/8}$ gives
 \[
 \sum_{\substack{w > y^{1/8}\\w\,\equiv\, j~({\rm mod}\, 2)}}
  \frac{\alpha_3}{2w^3}+O\left(\frac1{w^4}\right)
  =
  \sum_{k > y^{1/8}/2}\left( \frac{\alpha_3}{2(2k + j)^3} + O\Big(\frac{1}{k^4}\Big)\right)
 =\frac{\alpha_3}{8y^{1/4}}+O(\frac{1}{y^{3/8}}).\]

 We now consider the case $w\le y^{1/8}$.  For $p_8(v,w)>y$ with $1\le v<w$,
  it is necessary
 that $w>y^{1/8}/\sqrt{2}$.  For a given value of $w$ we sum on $v$,
 noting that $g(y^{1/4}/w^2)w\le v<w$.  Thus, we have the contribution
 \[
 \sum_{\substack{g(y^{1/4}/w^2)w\le v<w\\v\,\equiv \,i~({\rm mod}\, 2)}}
 \frac{1}{\sqrt{p_8(v,w)}}
% \sum_{\ell = g(\frac{y^{1/4}}{w^2})w/2}^{w/2}\frac{1}{\sqrt{p_8(2\ell+i,w)}}\]
%\[
%\int_{g(y^{1/4}/w^2)w}^w\frac1{\sqrt{p_8(v,w)}}\,dv+O\Big(\frac1{w^4}\Big)
 =\frac1{2w^3}\int_{g(y^{1/4}/w^2)}^1\frac1{\sqrt{p_8(u,1)}}\,du+O\Big(\frac1{w^4}\Big).
 \]
 We now sum this expression on $w$ with $y^{1/8}/\sqrt{2}< w\leq y^{1/8}$, $w\equiv j~({\rm mod}\, 2)$ getting
 \begin{align*}
 \sum_{\substack{y^{1/8}/\sqrt{2}< w\leq y^{1/8}\\w\,\equiv\, j~({\rm mod}\, 2)}}
& \left(
 \frac{1}{2w^3}\int
 _{g(y^{1/4}/w^2)}^1\frac1{\sqrt{p_8(u,1)}}\,du+O\Big(\frac1{w^4}\Big)\right)\\
 &=\frac12\int_{y^{1/8}/\sqrt{2}}^{y^{1/8}}\frac1{2x^3}\int_{g(y^{1/4}/x^2)}^1\frac1{\sqrt{p_8(u,1)}}
 \,dudx+O\Big(\frac1{y^{3/8}}\Big)\\
% =\sum_{k=y^{1/8}/2\sqrt{2}}^{y^{1/8}/2}
% \frac{1}{2(2k+j)^3} \left( \int_{g(y^{1/4}/(2k+j)^2)}^1\frac{1}{\sqrt{p_8(u,1)}}\, du
% +O\Big(\frac1{k^4}\Big)\right)
% \]
% \[
& =\frac{1}{8y^{1/4}}\int_1^2\int_{g(t)}^1\frac{1}{\sqrt{p_8(u,1)}}\,dudt + O\Big(\frac1{y^{3/8}}\Big),
 \end{align*}
%\[
% %\int_{y^{1/8}/\sqrt{2}}^{y^{1/8}}\frac1{w^3}\int_{g(y^{1/4}/w^2)}^1\frac1{\sqrt{p_8(u,1)}}\,du\,dw+O\big(y^{-3/8}\big)
% %=\frac12y^{-1/4}\int_1^2\int_{g(t)}^1\frac1{\sqrt{p_8(u,1)}}\,du\,dt+O\big(y^{-3/8}\big),
% \]
 where we use the substitution $t=y^{1/4}/x^2$.  This gives the contribution
 $\alpha_4/(8y^{1/4})+O(1/y^{3/8})$, completing the proof.
 \end{proof}

 \begin{cor}
 \label{cor:coprime}
 We have
 \[
  \sum_{\substack{1\leq v<w\\\gcd(v,w)=1\\p_8(v,w)>y\\(v,w)\in T_{ij}}}\frac1{\sqrt{p_8(v,w)}}=\frac{\alpha_3+\alpha_4}{6\zeta(2)y^{1/4}}+O\Big(\frac{\log y}{y^{3/8}}\Big).
  \]
  \end{cor}
  \begin{proof}
  First note that
  \[
  \sum_{\substack{1\leq v<w\\d\,\mid\,v,\,d\,\mid\,w}}\frac1{\sqrt{p_8(v,w)}}
=  \frac1{d^4}\sum_{1\leq v<w}\frac1{\sqrt{p_8(v,w)}}
  =O\Big(\frac1{d^4}\Big).
  \]
  Thus,
 % \[
%   \sum_{\substack{0<v<w\\w>y}}\frac1{\sqrt{p_8(v,w)}}
%   <\sum_{\substack{0<v<w\\w>y}}\frac1{w^4}
%   <\sum_{w>y}\frac1{w^3}=O\Big(\frac1{y^2}\Big).
%   \]
%   Thus,
 \[
    \sum_{\substack{1\leq v<w\\\gcd(v,w)=1\\p_8(v,w)>y\\(v,w)\in T_{ij}}}\frac1{\sqrt{p_8(v,w)}}=\Big(
   \sum_{\substack{d\le y^{1/8}\\d\, {\rm odd}}}  \sum_{\substack{1\leq v<w\\d\,\mid \,v,\,d\,\mid \,w\\p_8(v,w)>y\\(v,w)\in T_{ij}}}\frac{\mu(d)}{\sqrt{p_8(v,w)}}\Big) +O\Big(\frac1{y^{3/8}}\Big)
 \]
 (recall $\mu(d)$ denotes the M\"{o}bius function).
     By Proposition \ref{lem:tail}, the double sum here is
%   \begin{align*}
%\sum_{\substack{d\le y^{1/8}\\d\, {\rm odd}}}\sum_{\substack{1\leq v<w\\d\,|\,v,\,d\,|\,w\\p_8(v,w)>y\\(v,w)\in T_{ij}}}\frac{\mu(d)}{\sqrt{p_8(v,w)}}
%&=\sum_{\substack{d\le y^{1/8}\\d\, {\rm odd}}}\sum_{\substack{1\leq v<w\\p_8(v,w)>y/d^8\\(v,w)\in T_{ij}}}\frac{\mu(d)}{d^4\sqrt{p_8(v,w)}}\\
\[
\sum_{\substack{d\le y^{1/8}\\d\, {\rm odd}}}\frac{\mu(d)}{d^4}\Bigg(\frac{\alpha_3+\alpha_4}{8(y/d^8)^{1/4}}+O\Big(\frac1{(y/d^8)^{3/8}}\Big)\Bigg)\\
=\frac{\alpha_3+\alpha_4}{6\zeta(2)y^{1/4}}+O\Big(\frac{\log y}{y^{3/8}}\Big),
\]
completing the proof.
The last step used the fact that
for $p$ prime, $d>1$, and $x\ge1$, we have
\begin{equation}
\label{eq:fact5}
\sum_{n\le x,\,p\,\nmid\, n}\mu(n)/n^d=1/(\zeta(d)(1-p^{-d}))+O(1/((d-1)x^{d-1})).
\end{equation}
Indeed, the infinite sum has an Euler product which we
recognize as $1/(\zeta(d)(1-p^{-d}))$.
Further, the tail for $n>x$ converges absolutely to a sum that is $O(1/(d-1)x^{d-1})$.
\end{proof}

\begin{prop}
\label{lem:region}
Let $z\ge1$ and let $\mathcal{R}_2(z)$ be the region in the $x,y$ plane with
\[
0\le x\le y,\quad p_8(x,y)\le z,
\]
and let $A(z)$ be the area of $\mathcal{R}_2(z)$.  Then $A(z)=\beta z^{1/4}$, where
\[
\beta=\frac12\int_{\pi/4}^{\pi/2}\frac1{p_8(1,\tan\theta)^{1/4}\cos^2\theta}\,d\theta
=0.406683250144951\dots\,.
\]
Further, the projection of $\mathcal{R}_2(z)$ on any line has length $O(z^{1/8})$.
\end{prop}
\begin{proof}
We have
\[
A(z)=\int_{\pi/4}^{\pi/2}\int_0^Br\,dr\,d\theta,
\]
where $B$ is the distance to the origin of the point on $p_8(x,y)=z$ with
$\theta=\arctan(y/x)$.  We have at this point that $z=p_8(x,y)=p_8(1,\tan\theta)x^8$
and $y=x\tan\theta$.  Thus,
\[
B=(x^2+y^2)^{1/2}=x(1+\tan^2\theta)^{1/2}=(z/p_8(1,\tan\theta))^{1/8}(1+\tan^2\theta)^{1/2}=\frac{z^{1/8}}{p_8(1,\tan\theta)^{1/8}\cos\theta}.
\]
Thus,
\[
A(z)=\int_{\pi/4}^{\pi/2}\frac12B^2\,d\theta
=\frac12z^{1/4}\int_{\pi/4}^{\pi/2}\frac{1}{p_8(1,\tan\theta)^{1/4}\cos^2\theta}\,d\theta,
\]
completing the proof.
The final assertion of the lemma follows from the fact that all points $(x,y)\in \mathcal{R}_2(z)$ satisfy ${\rm max}\{ |x|,|y|\} \leq z^{1/8}$.
\end{proof}

\begin{cor}
\label{cor:PL}
Let $z\ge1$ and let $L(z)$ denote the number of lattice points $(v,w)$ in $\mathcal{R}_2(z)$
with $(v,w)\in T_{i,j}$
and let $L'(z)$ be the number of them with
$\gcd(v,w)=1$.  Then
\[
L(z)=\frac{\beta}{4} z^{1/4}+O(z^{1/8})~\hbox{ and }~
L'(z)=\frac\beta{3\zeta(2)}z^{1/4}+O(z^{1/8}\log z).
\]
\end{cor}
\begin{proof}
The first assertion is immediate from the Principle of Lipschitz.  For the second
assertion, let $L_d(z)$ be the number of lattice points $(dv,dw)\ne(0,0)$ in
$\mathcal{R}_2(z)$, where $d$ is a positive integer.   We have
\[
L'(z)=\sum_{d\,{\rm odd}}\mu(d)L_d(z)=\sum_{\substack{d<z^{1/8}\\d\, {\rm odd}}}\mu(d)L_d(z),
\]
since if $d\ge z^{1/8}$, then $L_d(z)=0$.
Note that $L_d(z)$ is the number of lattice points in $\mathcal{R}_2(z/d^8)\setminus\{(0,0)\}$.
So, by
Proposition \ref{lem:region}, \eqref{eq:fact5}, and the Principle of Lipschitz,
\[
L'(z)=\frac{\beta}{4} z^{1/4}\sum_{\substack{d< z^{1/8}\\d\, {\rm odd}}}\frac{\mu(d)}{d^2}
+O\Big(z^{1/8}\sum_{\substack{d<z^{1/8}\\d, {\rm odd}}}\frac1d\Big)
=\frac{\beta}{3\zeta(2)}z^{1/4}+O\big(z^{1/8}\log x\big).
\]
\end{proof}

\subsection{Facts from analytic number theory}
\label{analytic nt}
We now present some useful facts from analytic number theory.
 Let
\[
M(x)=\sum_{n\le x}\mu(n),\quad Q(x)=\sum_{n\le x}\mu(n)^2,
\quad Z(x)=\sum_{n\le x}\frac{\mu(n)}n,\quad S(x)=\sum_{n\le x}\frac{\mu(n)}{\sqrt{n}}.
\]
%\textcolor{blue}{I added a new fact as we use it several times.}
\begin{prop}
\label{lem:useful}
We have the following inequalities: As $x\to\infty$,
\begin{itemize}
\item[(1)]
$|M(x)|\le x/\exp((\log x)^{3/5+o(1)})$,
\item[(2)]
$|Q(x)-x/\zeta(2)|\le x^{1/2}/\exp((\log x)^{3/5+o(1)})$,
\item[(3)]
%$\Big|\sum_{n\le x}\mu(n)/n\Big|
$|Z(x)|
\le\exp(-(\log x)^{3/5+o(1)})$, and
\item[(4)]
%$\Big|\sum_{n\le x}\mu(n)^2/\sqrt{n}-(2/\zeta(2))\sqrt{x}\Big|
$|S(x)-(2/\zeta(2))\sqrt{x}|
\le\exp(-(\log x)^{3/5+o(1)})$.
\end{itemize}
\end{prop}
%\textcolor{blue}{I'm not used to lemmas having corollaries. If you say it's OK, I accept that. Otherwise, should this be a proposition?}
\begin{proof}
Facts (1) and (2) may be found in Walfisz \cite[p. 146]{Wa}.

For fact (3), first note that from the prime number theorem,
$\sum_{n\ge1}\mu(n)/n=0$, see \cite[p. 43]{Ten}.
%by the
%Hardy--Littlewood--Karamata theorem (see Tenenbaum \cite[p. 227]{Ten}),
%\[
%\sum_{n=1}^\infty\frac{\mu(n)}n=0.
%\]
Thus, by partial (Abel) summation,
\[
%\sum_{n\le x}\frac{\mu(n)}n
Z(x)=-\sum_{n>x}\frac{\mu(n)}n
=-\int_x^\infty\frac{M(t)-M(x)}{t^2}\,dt
\]
and the result follows from fact (1).
%which converges to some constant $\kappa$.  Further, by the
%Hardy--Littlewood--Karamata theorem (see Tenenbaum \cite[p. 227]{Ten}),
%we have $\kappa=0$.  Thus,
%\[
%\sum_{n\le x}\frac{\mu(n)}n=\frac1xM(x)+\int_1^x\frac{M(t)}{t^2}\,dt
%=\frac{M(x)}x+\kappa-\int_x^\infty\frac{M(t)}{t^2}\,dt
%=O(x/\exp((\log x)^{1/2})),
%\]
%using fact (1).

For fact (4), let
$\epsilon>0$ be arbitrarily small, and fixed.
Let $y=\exp((\log x)^{3/5-\epsilon})$ and let $z=(x/y)^{1/2}$.  We have
\begin{align*}
%\sum_{n\le x}\frac{\mu(n)^2}{\sqrt{n}}
S(x)&=\sum_{n\le x}\sum_{d^2\,\mid\,n}\frac{\mu(d)}{\sqrt{n}}
=\sum_{ad^2\le x}\frac{\mu(d)}{d\sqrt{a}}\\
&=\sum_{a\le y}\sum_{d\le\sqrt{x/a}}\frac{\mu(d)}{d\sqrt{a}}+
\sum_{d\le z}\sum_{a\le x/d^2}\frac{\mu(d)}{d\sqrt{a}}-
\sum_{a\le y}\sum_{d\le z }\frac{\mu(d)}{d\sqrt{a}}=s_1+s_2-s_3,
\end{align*}
say.  Using $\sum_{a\le y}1/\sqrt{a}=O(\sqrt{y})$, we have, as $x\to\infty$,
\[
|s_1|\le \exp(-(\log x)^{3/5+o(1)}),
\]
from fact (3).
Almost the same calculation works for $s_3$.
Note that
\[
\sum_{n\le x}1/\sqrt{n}=2\sqrt{x}+\zeta(1/2)+O(1/\sqrt{x}), \hbox{ when }x\ge1,
\]
see Apostol \cite[Theorem 3.2 (b)]{Ap}.  Thus,
%Also, using fact (4),
\begin{align*}
s_2
&=\sum_{d\le z}\frac{\mu(d)}{d}\big(2\sqrt{x/d^2}+\zeta(1/2)+O(1/\sqrt{x/d^2})\big)\\
&=2\sqrt{x}\sum_{d\le z}\frac{\mu(d)}{d^2}+\Bigg(\zeta(1/2)\sum_{d\le z}\frac{\mu(d)}{d}\Bigg)
+O(z/\sqrt{x}).
\end{align*}
Note that $z/\sqrt{x}=\exp(-\frac12(\log x)^{3/5-\epsilon})$.
We shall use fact (3) on the second sum here, and for the first sum,
\[
\sum_{d\le z}\frac{\mu(d)}{d^2}=\frac1{\zeta(2)}-\sum_{d>z}\frac{\mu(d)}{d^2}
=\frac1{\zeta(2)}-\int_z^\infty(M(t)-M(z))\cdot\frac2{t^3}\,dt.
\]
We can estimate the integral using fact (1), so that
\[
|s_2-2\sqrt{x}/\zeta(2)|=O\Big( \exp\big(-\frac12(\log x)^{3/5-\epsilon}\big)\Big).
\]
Since $\epsilon>0$ is arbitrary, this
completes the proof of fact (4) and the proposition.
\end{proof}

As the definition of $A(r,v,w)$ depends on whether or not $3\mid r$, we must adapt
Proposition~\ref{lem:useful} to each of these two cases.  It suffices to consider
the case $3\mid r$, since then the case $3\nmid r$ can be found by subtracting
from the full sum.  Let
\[
M_3(x)=\sum_{\substack{n\le x\\3\,\mid\,n}}\mu(n),
\quad Q_3(x)=\sum_{\substack{n\le x\\3\,\mid\,n}}\mu(n)^2,\quad
Z_3(x)=\sum_{\substack{n\le x\\3\,\mid\,n}}\frac{\mu(n)}n,\quad
S_3(x)=\sum_{\substack{n\le x\\3\,\mid\,n}}\frac{\mu(n)^2}{\sqrt{n}}.
\]
\begin{cor}
\label{cor:divby3}
We have the following inequalities:  As $x\to\infty$,
\begin{enumerate}
\item
$|M_3(x)|\le x/\exp((\log x)^{3/5+o(1)})$,
\item
$|Q_3(x)-x/(4\zeta(2))|\le x^{1/2}/\exp((\log x)^{3/5+o(1)})$,
\item
%$\Big|\sum_{n\le x,\,3\,\mid\,n}\mu(n)/n\Big|
$|Z_3(x)|
\le\exp(-(\log x)^{3/5+o(1)})$, and
\item
%$\Big|\sum_{n\le x,\,3\,|\,n}\mu(n)^2/\sqrt{n}-\sqrt{x}/(2\zeta(2))\Big|
$|S_3(x)-\sqrt{x}/(2\zeta(2))|
\le\exp(-(\log x)^{3/5+o(1)})$.
\end{enumerate}
\end{cor}
\begin{proof}
First note that
\begin{equation}
\label{eq:id1}
M_3(x)=-\sum_{j\ge1}M(x/3^j).
\end{equation}
Indeed, this holds trivially for $x<3$.  Assume  it holds for
$x<3^k$.  Then for $x<3^{k+1}$, we have
\[
M_3(x)=\sum_{m\le x/3}\mu(3m)=-\sum_{\substack{m\le x/3\\3\,\nmid\,m}}\mu(m)
=-M(x/3)+M_3(x/3),
%=-M(x/3)-\sum_{j\ge1}M((x/3)/3^j),
\]
and so \eqref{eq:id1} follows by using the induction hypothesis on $M_3(x/3)$.
Next, \eqref{eq:id1} implies that
\[
M_3(x)=-\sum_{\substack{j\ge1\\3^j\le\sqrt{x}}}M(x/3^j)-\sum_{3^j>\sqrt{x}}M(x/3^j).
\]
For the first sum we use (1) of Propostion \ref{lem:useful} on each term,
and for the second sum we use the trivial bound $|M(x/3^j)|\le x/3^j$, thus
establishing part (1) of the corollary.

 Similar induction proofs show that
\[
Q_3(x)=\sum_{j\ge1}(-1)^{j-1}Q(x/3^j), \quad
Z_3(x)=-\sum_{j\ge1}\frac1{3^j}Z(x/3^j),\quad
S_3(x)=\sum_{j\ge1}\frac{(-1)^{j-1}}{3^{j/2}}S(x/3^j),
\]
and so parts (2)--(4) follow from parts (2)--(4) of Proposition \ref{lem:useful}.
%The argument in \cite{Wa} for (1) and (2) in Proposition \ref{lem:useful}
%depends in essence on the best
%known zero-free region in the ciritical strip for $\zeta(s)$ (due to Vinogradov
%and Korobov).  To handle the case when $3\mid n$ one merely has to
%replace $\zeta(s)$ with $3^{-s}\zeta(s)$, which evidently has the same zeros
%as $\zeta(s)$.
%For (3), we follow the proof of (3) in Proposition \ref{lem:useful},
%using $M_3(x)$ instead of $M(x)$.  We similarly have (4).
\end{proof}
\noindent{\bf Remark}.
Thanks are due to Paul Pollack who suggested the identity \eqref{eq:id1}.

\subsection{Breaking the region into four cases}
\label{The six cases}
The different formulae for $A(r,v,w)$ depending on whether $3 \mid r$
and the values of $v$ and $w$ modulo 2 suggest that we break the computation into four cases:

\begin{itemize}
%\item[(i)] $3\nmid  r$, $v\equiv 1~({\rm mod}\, 2)$, $w\equiv 0~({\rm mod}\, 2)$,

\item[(i)] $3\nmid  r$, $v\not\equiv w~({\rm mod}\, 2)$,

%\item[(ii)] $3\nmid  r$, $v\equiv 0~({\rm mod}\, 2)$, $w\equiv 1~({\rm mod}\, 2)$,

%\item[(iii)] $3\mid r$, $v\equiv 1~({\rm mod}\, 2)$, $w\equiv 0~({\rm mod}\, 2)$,

\item[(ii)] $3\mid r$, $v\not\equiv w~({\rm mod}\, 2)$,

%\item[(iv)] $3 \mid r$, $v\equiv 0~({\rm mod}\, 2)$, $w\equiv 1~({\rm mod}\, 2)$,

\item[(iii)] $3\nmid r$, $2\nmid vw$, and

\item[(iv)] $3 \mid r$, $2\nmid vw$.
\end{itemize}

We give the details in the first case, and show how the constants change in
the subsequent cases.
Let $N^{\rm (i)}(X)$ be the number of $(r,v,w)$ in the first of the six cases with
$|A(r,v,w)|$ $= 27r^2p_8(v,w)$ $\leq A_b$, that is,
$r^2 p_8(v,w)\leq \frac{1}{4^{1/3}27}X^{1/3}$.
%Recall we have $r$ squarefree and positive, and in this case,  $3\nmid r$. We also %have $1\leq v < w$, gcd$(v,w)=1$,  $v\equiv 1~({\rm mod}\, 2)$,
%and $w\equiv 0~({\rm mod}\, 2)$.
Define
\[
\calS_{0} := \{ (v,w)\; | \; v\not\equiv w\kern-8pt\pmod{2},\,
1\leq v < w,\,\gcd(v,w)=1\},
\quad
s_{0} := \sum_{\calS_{0}}\frac{1}{\sqrt{p_8(v,w)}}.
\]

Let $N^{({\rm i})}(Y,X)$ be the number of triples $r,v,w$ counted by $N^{({\rm i})}(X)$ with $p_8(v,w)\le Y$,
let $N^{({\rm i})}_Z(Y,X)$ be the number of these triples where also $r\le Z$, and let
$N^{({\rm i})}_Z(X)$ be the number of triples with $r\le Z$ and $p_8(v,w)$
unrestricted.  Then, if $Z^2Y=\eta X^{1/3}$, for $\eta = \frac{1}{4^{1/3}27}$, we have
\[
N^{({\rm i})}(X)=N^{({\rm i})}(Y,X)+N^{({\rm i})}_Z(X)-N^{({\rm i})}_Z(Y,X).
\]
We will choose $Z=X^\delta$, where $\delta>0$ is fairly small.
We then have $Y=\eta X^{1/3 - 2\delta}$.
 It is convenient
to use the notation $f(x)=g(x)+O_\pm(h(x))$ if $|f(x)-g(x)|\le h(x)$.

\bigskip
\noindent{\bf The calculation of $N^{({\rm i})}(Y,X)$}
\medskip

Let $Q_3'(x)=Q(x)-Q_3(x)$, so that from Proposition \ref{lem:useful}
and Corollary \ref{cor:divby3} we have
\begin{equation}
\label{eq:Q3'}
Q_3'(x)=\frac3{4\zeta(2)}x+R_3'(x),
\hbox{ where }R_3'(x)=O_\pm(x^{1/2}/\exp((\log x)^{3/5+o(1)}))
\end{equation}
as $x\to\infty$.  Thus,
\begin{align}
\label{eq:NYX}
N^{({\rm i})}(Y,X)&=\sum_{\substack{(v,w)\in \calS_{0}\\p_8(v,w)\le Y}}\;\sum_{\substack{r\le \eta^{1/2} X^{1/6}/\sqrt{p_8(v,w)}\\r~{\rm squarefree},\, 3\,\nmid\,r}}1\notag\\
&=\frac{3\eta^{1/2}}{4\zeta(2)}X^{1/6}\sum_{\substack{(v,w)\in \calS_{0}\\p_8(v,w)\le Y}}
\frac1{\sqrt{p_8(v,w)}}+\sum_{\substack{(v,w)\in \calS_{0}\\p_8(v,w)\le Y}}R_3'(\eta^{1/2} X^{1/6}/\sqrt{p_8(v,w)}).
\end{align}
%from Corollary~\ref{Squarefree3}.
Since
$\eta^{1/2} X^{1/6}/\sqrt{p_8(v,w)}\ge \eta^{1/2}X^{1/6}/Y^{1/2}=\eta^{1/2}Z=\eta^{1/2}X^\delta$,
%Corollary~\ref{Squarefree3} also implies
the remainder term has absolute value at most
\[
\frac{X^{1/12}}{\exp((\log X)^{3/5+o(1)})}\sum_{\substack{(v,w)\in \calS_{0}\\p_8(v,w)\le Y}}\frac1{p_8(v,w)^{1/4}}
=\frac{X^{1/12}}{\exp((\log X)^{3/5+o(1)})}
\]
as $X\to\infty$,  since the sum here is $O(\log Y)$.

The main term for $N^{({\rm i})}(Y,X)$ is
\[
\frac{3\eta^{1/2}}{4\zeta(2)}X^{1/6}\Bigg(\sum_{\substack{(v,w)\in \calS_{0}}}\frac1{\sqrt{p_8(v,w)}}-\sum_{\substack{(v,w)\in \calS_{0}\\p_8(v,w)>Y}}\frac1{\sqrt{p_8(v,w)}}\Bigg).
\]
The first sum is the constant $s_{0}$.  The second sum is estimated in
Corollary \ref{cor:coprime} for the cases $i,j$ being $0,1$ and $1,0$.
Thus, we have, as $X\to\infty$,
\[
N^{({\rm i})}(Y,X)=
\frac{3\eta^{1/2}}{4\zeta(2)}s_{0}X^{1/6}
-
\frac{\eta^{1/2}(\alpha_3+\alpha_4)}{4\zeta(2)^2}\frac{X^{1/6}}{Y^{1/4}}
+O\Big(\frac{X^{1/6}\log Y}{Y^{3/8}}\Big)+O_\pm\Big(\frac{X^{1/12}}{\exp((\log X)^{3/5+o(1)})}\Big).
\]
We shall take $\delta<1/18$, and so the first error term above is absorbed into
the second, and we have, as $X\to\infty$,
\begin{equation*}
%\label{eq:N(i)(Y,X)}
N^{({\rm i})}(Y,X)=
\frac{3\eta^{1/2}}{4\zeta(2)}s_{0}
X^{1/6}-
\frac{\eta^{1/2}(\alpha_3+\alpha_4)}{4\zeta(2)^2}
X^{1/12+\delta/2}
+O_\pm\Big(\frac{X^{1/12}}{\exp((\log X)^{3/5+o(1)})}\Big).
\end{equation*}

It is convenient to compute the infinite sum $s_{0}$ numerically without
the coprimality condition for $v$ and $w$, so let
\[
s_{0}'=\sum_{\substack{1\leq v<w\\v\,\not\equiv\,w\kern-5pt\pmod{2}}}\frac{1}{\sqrt{p_8(v,w)}}
=.0646797\dots\,\]
(as in Section~\ref{Constants}).
%=.0543946\dots\,. \]
Note that if $v=dv_0$ and $w=dw_0$, for some  $d$, then
${\sqrt{p_8(v,w)}}$ $=
{d^4\sqrt{p_8(v_0,w_0)}}$.
%The only prime that can not divide both $v$ and $w$ in Case i is 2.
So
\[
s_{0}
=
 s_{0}'\sum_{d\, {\rm odd}}\frac{\mu(d)}{d^4}.
\]
 From \eqref{eq:fact5}, we have $s_{0}=
\frac{16}{15\zeta(4)}s'_{0}$.

%Thus, as $X\ra \infty$ we have
%\[
%N^{({\rm i})}(Y,X)=
%\frac{3}{4}\cdot \frac{1}{2^{1/3}27^{1/2}\zeta(2)}\cdot \frac{16}{15\zeta(4)}s'_{10}
%X^{1/6}-
%\frac{\eta^{1/2}(\alpha_1+\alpha_2)}{8\zeta(2)^2}
%X^{1/12+\delta/2}
%\]
%\[
%+O_\pm(X^{1/12}/\exp((\log X)^{3/5+o(1)})).
%\]
%\textcolor{blue}{There are two inconsistencies in the above display. First, I break
%up the coefficient of $X^{1/6}$ as a product, in order to better see where
%the analogous coefficients of $N^{(ii)}(X), \ldots , N^{(vi)}(X)$ come from
%in the next section.
%Second, I leave the $\eta$ in the coefficient of $X^{1/12+\delta/2}$ to convince
%the reader, in the next section, that after assembly, the coefficient of $X^{1/12+\delta/2}$
%will be 0 in all six cases. However, the inconsistencies are not nice. I'm open
%to suggestions of how to handle this.}

\bigskip
\noindent{\bf The calculation of $N^{({\rm i})}_Z(Y,X)$}
\medskip

The calculation of $N^{({\rm i})}_Z(Y,X)$ parallels that of $N^{({\rm i})}(Y,X)$.  In place of \eqref{eq:NYX}
we have
\[
N^{({\rm i})}_Z(Y,X)
=\Big(\frac3{4\zeta(2)}Z+R_3'(Z)\Big)\sum_{\substack{(v,w)\in \calS_{0}\\p_8(v,w)\le Y}}1.
\]
%from Corollary~\ref{Squarefree3}.
By Corollary \ref{cor:PL} in the cases $i,j$ being $0,1$ and $1,0$, the sum here is
$2\beta Y^{1/4}/(3\zeta(2))+O(Y^{1/8}\log Y)$.
By \eqref{eq:Q3'} we thus have, as $X\to\infty$,
\[
N^{({\rm i})}_Z(Y,X)
=
\frac{\eta^{1/4}\beta}{2\zeta(2)^2} X^{1/12 + \delta/2}+O_\pm(X^{1/12}/\exp((\log X)^{3/5+o(1)})).
\]

\bigskip
\noindent{\bf The calculation of $N^{({\rm i})}_Z(X)$}
\medskip

We have
\[
N^{({\rm i})}_Z(X)=\sum_{\substack{r\le Z\\r~{\rm squarefree}\\3\,\nmid\, r}}\sum_{\substack{(v,w)\in \calS_{0}\\p_8(v,w)\le \eta X^{1/3}/r^2}}1
=\sum_{\substack{r\le Z\\r~{\rm squarefree}\\ 3\,\nmid\, r}}\Bigg(\frac{2\eta^{1/4}\beta X^{1/12}}{3\zeta(2)r^{1/2}}+O\Big(\frac{X^{1/24}\log X}{r^{1/4}}\Big)\Bigg),
\]
using Corollary \ref{cor:PL} for $i,j$ being $0,1$ and $1,0$.
The remainder term here is $O(X^{1/24}Z^{3/4}\log X)$,
which is negligible.  For the main term we need to sum $1/r^{1/2}$ fairly
precisely, which follows from Proposition \ref{lem:useful} and Corollary \ref{cor:divby3}.
%Corollary~\ref{SumRecipRootNo3}.
So, as $X\to\infty$,
\[
N^{({\rm i})}_Z(X)=\frac{\eta^{1/4}\beta}{\zeta(2)^2}
X^{1/12+\delta/2}
%-\frac{\beta}{\zeta(2)}\Big(c_1-\frac1{\zeta(2)}\Big)X^{1/12}
+O_\pm(X^{1/12}/\exp((\log X)^{3/5+o(1)})).
\]

\subsection{The main theorem for two pairs}
\label{The main theorem}
%\textcolor{blue}{Or would you prefer this section be called {\it Assembling}?}
%
%\textcolor{blue}{The next two paragraphs are new.}
Recall $N^{({\rm i})}(X)$ is the number of $(r,v,w)$ in the first of the four cases
described at the beginning of Section~\ref{The six cases}, such that $|A(r,v,w)|\leq A_b$.
And recall $N^{({\rm i})}(X) = N^{({\rm i})}(Y,X)+N^{({\rm i})}_Z(X)-N^{({\rm i})}_Z(Y,X)$, which is
\[
\frac{3}{4}{\cdot }\frac{1}{2^{1/3}27^{1/2}\zeta(2)}{\cdot }\frac{16}{15\zeta(4)}s_{0}
X^{1/6}
+\frac{\eta^{1/4}}{4\zeta(2)^2}
(2\beta {-} \alpha_3 {-} \alpha_4)
X^{1/12+\delta/2}
+O_\pm\Big(\frac{X^{1/12}}{\exp((\log X)^{3/5+o(1)})}\Big)
\]
as $X\to \infty$.  This holds for any value of $\delta$ with $0<\delta<\frac1{18}$. Yet
if $2\beta\neq \alpha_3+\alpha_4$, then the above statement cannot be true for more than one value of $\delta$. So $2\beta = \alpha_3+\alpha_4$ and, as $X\to\infty$,
\[
N^{({\rm i})}(X)=\frac{3}{4}\cdot \frac{1}{2^{1/3}27^{1/2}\zeta(2)}\cdot \frac{16}{15\zeta(4)}s_{0}
X^{1/6}+O_\pm(X^{1/12}/\exp((\log X)^{3/5+o(1)})).\]
We can also prove $2\beta = \alpha_3+\alpha_4$ directly.
We have
\[
\alpha_3 + \alpha_4
=\int_0^1\int_0^{p_8(u,1)^{-1/2}}\int_0^{p_8(u,1)^{1/4}}\, dtdzdu
=\int_0^1\int_0^{p_8(u,1)^{-1/4}}dadu.
\]
Make the change of variables $a=x^2$, $u=y/x$.
The Jacobian has value 2. The region described by the limits of integration
of the last double integral is transformed to the region
in the first quadrant of the $xy$-plane bounded by $p_8(x,y)=1$, $y=0$,
and $y=x$. The latter region has the same area as ${\mathcal R}_2(1)$, i.e. $\beta$.

The remaining three cases are similar.
%The coefficient of $X^{1/12 + \delta/2}$
%is $\frac{\chi \eta^{1/4}}{8\zeta(2)^2}(2\beta - \alpha_1-\alpha_2))$
%where $\eta = \frac{1}{4^{1/3}\eta'}$ with $\eta' = 27, 27, \frac{1}{3},\frac{1}{3},
%\frac{27}{16}, \frac{1}{48}$, and $\chi = 1, 1, \frac{1}{3}, \frac{1}{3},
%1, \frac{1}{3}$ in Cases $({\rm i}), \ldots , (vi)$, respectively.
%For the values of $\eta'$ see the coefficient of  $r^2 p_8(v,w)$ in the definition
%of $A(r,v,w)$ before Lemma~\ref{TwoPairsIntegerModel} and for the values of $\chi$ see
%Corollaries~\ref{Squarefree3} and \ref{SumRecipRootNo3}. We see in each
%case, since $2\beta = \alpha_1+\alpha_2$, the coefficient of
%$X^{1/12 + \delta/2}$ is 0.
%
Let
%$S_{01}$ be the same as $S_{10}$ but with $v$ even and $w$ odd, and  let
\[
%s_{01}'=\sum_{\substack{1\leq v<w\\2\mid v,\,2\,\nmid \,w}}\frac1{\sqrt{p_8(v,w)}}=0.0102850\dots,
%\quad
s_{1}'=\sum_{\substack{1\leq v<w\\2\,\nmid \,vw}}\frac1{\sqrt{p_8(v,w)}}=0.0161699\dots\,
\]
%  Also, let $S_{11}$ be the same as $S_{10}$, but
%with $v$ and $w$ both odd, and let
%\[
%s_{11}'=\sum_{0<v<w,\,2\,\nmid \,vw}\frac1{\sqrt{p_8(v,w)}}=0.0161699\dots\,.
%\]
(as in Section~\ref{Constants}).
If $N^{(j)}(X)$ is the count of the number of $(r,v,w)$ with $A(r,v,w)\leq A_b$ in Case ($j$)
of the four cases,
then as $X\to \infty$ we have
\begin{align*}
%\begin{array}{l}
%N^{\rm (ii)}(X) &= \frac{3}{4}\cdot \frac{1}{2^{1/3}27^{1/2}\zeta(2)}
%\cdot \frac{16}{15\zeta(4)}s_{01}'X^{1/6}+O_\pm(X^{1/12}/\exp((\log X)^{3/5+o(1)})),
%\approx 0.001882X^{1/6}+O(X^{1/8})
%\\
N^{\rm (ii)}(X) &=\frac{1}{4}{\cdot}
\frac{1}{2^{1/3}(1/3)^{1/2}\zeta(2)}
{\cdot} \frac{16}{15\zeta(4)}s_{0}'X^{1/6}+O_\pm(X^{1/12}/\exp((\log X)^{3/5+o(1)})),
%\approx 0.044802X^{1/6}+O(X^{1/8})
\\
%N^{\rm (iv)}(X)  &=
%\frac{1}{4}\cdot
%\frac{1}{2^{1/3}(1/3)^{1/2}\zeta(2)}
%\cdot \frac{16}{15\zeta(4)}s_{01}'X^{1/6}+O_\pm(X^{1/12}/\exp((\log X)^{3/5+o(1)})),
%\approx 0.008471X^{1/6}+O(X^{1/8})
%\\
N^{\rm (iii)}(X)  &=
\frac{3}{4}{\cdot}
\frac{1}{2^{1/3}(27/16)^{1/2}\zeta(2)}
{\cdot }\frac{16}{15\zeta(4)}s_{1}'X^{1/6}+
O_\pm(X^{1/12}/\exp((\log X)^{3/5+o(1)})),\, {\rm and}
%\approx 0.011838X^{1/6}+O(X^{1/8})
\\
N^{\rm (iv)}(X)  &=
\frac{1}{4}{\cdot}
\frac{1}{2^{1/3}(1/48)^{1/2}\zeta(2)}
{\cdot} \frac{16}{15\zeta(4)}s_{1}'X^{1/6}+O_\pm(X^{1/12}/\exp((\log X)^{3/5+o(1)})).
%\approx 0.053273X^{1/6}+O(X^{1/8}).
%\end{array}\]
\end{align*}
Recall (from Section~\ref{Constants}) that
 \[
c_{2,1}=\frac{16}{2^{1/3}27^{1/2}5\zeta (2)\zeta (4)}
\left( s_{0}'+ 4s_{1}'\right).
\]
%Also recall that $N_2(X)$ is the number of elliptic curves over $\Q$
%with two pairs of Galois-stable cyclic subgroups of
%order $4$ and height at most $X$.

\begin{theorem}
\label{Count2pairs}
The number of $E/\Q$ of height at most $X$ with two pairs
of Galois-stable
cyclic subgroups is
$N_2(X) = c_{2,1}X^{1/6}+O_\pm(X^{1/12}/\exp((\log X)^{3/5+o(1)}))$ as $X\to \infty$,
where $c_{2,1}= 0.0355154\ldots$ is a calculable constant.
\end{theorem}

\begin{proof}
This follows from Proposition~\ref{BijectionAVWandTwoPairs}
and the computations above.
\end{proof}

We remark that if the Riemann Hypothesis is assumed we can obtain
power-saving error estimates over those recorded in Proposition \ref{lem:useful}
and Corollary \ref{cor:divby3}
%(apart from fact (4) \textcolor{blue}{How about the new fact (5)?}),
and so obtain an error estimate for $N_2(X)$ that shaves
a constant off of the exponent $1/12$.

\subsection{Numerical evidence - two pairs}
\label{Numerical evidence}

In this section we give numerical evidence for
Theorem~\ref{Count2pairs}.
Here we use $N_2(X)$ to denote the number of triples $(r,v,w)$
with $r$ a squarefree positive integer, and $v, w\in \Z$  with
$1\leq v < w$ and gcd$(v,w)=1$, such that
$|A(r,v,w)|\leq A_b = \frac{1}{4^{1/3}}X^{1/3}$
and $A(r,v,w)$ as defined right before
Lemma~\ref{TwoPairsIntegerModel}.
In the table below we present $N_2(X)$ and
$N_2(X)-c_{2,1}X^{1/6}$ for various values of $X$. In the last column,
we round to one digit past the decimal point.
We have $c_{2,1}=0.035515447977\ldots\, $.
\medskip

\centerline{
\renewcommand{\arraystretch}{1.5}
\begin{tabular}{|l|r|r|r|}\hline
$X$ & $N_2(X)\hskip.4cm$  & $N_2(X) - c_{2,1}X^{1/6}$
\\
\hline
$10^{30}$ & $3544$  & $-7.5\hskip.75cm$ \\
$10^{36}$ & $35486$  & $-29.4\hskip.75cm$ \\
$10^{42}$ & $355140$  &  $-14.5\hskip.75cm$ \\
$10^{48}$ & $3551596$  &  $ 51.2\hskip.75cm$ \\
$10^{54}$ & $35515580$  &  $132.0\hskip.75cm$ \\
$10^{60}$ & $355154548$  & $68.2\hskip.75cm$ \\
\hline
\end{tabular}
}


\begin{thebibliography}{19}


\bibitem[Ap]{Ap}
Apostol, T. M., {\it Introduction to analytic number theory},
Springer, New York, 1976.

\bibitem[CKV]{CKV}
Cullinan, J., Kenney, M., Voight, J., {\it
On a probabilistic local-global principle for torsion on elliptic curves},
in preparation.

\bibitem[HS]{HS}
Harron, R., Snowden, A., {\it
Counting elliptic curves with prescribed torsion},
J. Reine Angew. Math. {\bf 729}, (2017), 151--170.

\bibitem[Hu]{Hu}
Huxley, M.N., {\it Exponential sums and lattice points III},
Proc. London Math. Soc. {\bf 87} (2003), no.\ 3, 591--609.

\bibitem[PPV]{PPV}
Pizzo, M., Pomerance, C., Voight, J., {\it Counting elliptic curves with an isogeny of degree three}, Proc. Amer. Math. Soc., Ser. B, {\bf 7} (2020) 28--42.

\bibitem[Sc]{Sc}
Schaefer, E.F., {\it Class groups and Selmer groups}, J. Number Theory {\bf 56} (1996), no. 1, 79--114.

\bibitem[Te]{Ten}
Tenenbaum, G., {\it Introduction to analytic and probabilistic number theory},
Cambridge University Press, Cambridge, 1995.

\bibitem[Wa]{Wa}
Walfisz, A., {\it Weylsche Exponentialsummen in der neueren Zahlentheorie},
Mathematische Forschungsberichte, XV. VEB Deutscher Verlag der Wissenschaften,
Berlin, 1963.



\end{thebibliography}
\end{document}